\documentclass[10pt]{article}

\usepackage{amsmath,amsthm}
\usepackage{amssymb}
\usepackage{color}
\usepackage{breqn}
\usepackage{enumerate}
\usepackage{graphicx}
\usepackage{bbm}
\usepackage[affil-it]{authblk}

\newtheorem{theorem}{Theorem}
\newtheorem{corollary}[theorem]{Corollary}
\newtheorem{lemma}[theorem]{Lemma}
\newtheorem{proposition}[theorem]{Proposition}
\newtheorem{question}[theorem]{Question}

\newtheorem{conjecture}[theorem]{Conjecture}

\newtheorem{observation}[theorem]{Observation}
\theoremstyle{definition}

\newcommand{\diam}{\operatorname{diam}}

\makeatletter
\def\blfootnote{\gdef\@thefnmark{}\@footnotetext}
\makeatother	
	
\begin{document}

\title{Covering complete graphs by monochromatically bounded sets}
\date{}
\author{Luka Mili\'{c}evi\'{c}}
\maketitle

\abstract{Given a $k$-colouring of the edges of the complete graph $K_n$, are there $k-1$ monochromatic components that cover its vertices? This important special case of the well-known Lov\'asz-Ryser conjecture is still open. In this paper we consider a strengthening of this question, where we insist that the covering sets are not merely connected but have bounded diameter. In particular, we prove that for any colouring of $E(K_n)$ with 4 colours, there is a choice of sets $A_1, A_2, A_3$ that cover all vertices, and colours $c_1, c_2, c_3$, such that for each $i = 1,2,3$ the monochromatic subgraph induced by the set $A_i$ and the colour $c_i$ has diameter at most 160.}

\section{Introduction}

Given a graph $G$, whose edges are coloured with a colouring $\chi\colon E(G) \to C$ (where adjacent edges are allowed to use the same colour), given a set of vertices $A$, and a colour $c \in C$, we write $G[A, c]$ for the subgraph induced by $A$ and the colour $c$, namely the graph on the vertex set $A$ and the edges $\{xy\colon x,y \in A, \chi(xy) = c\}$. In particular, when $A = V(G)$, we write $G[c]$ instead of $G[V(G), c]$. Finally, we also use the usual notion of the induced subgraph $G[A]$ which is the graph on the vertex set $A$ with edges $\{xy\colon x,y \in A, xy \in E(G)\}$. We usually write $[n]=\{1,2,\dots,n\}$ for the vertex set of $K_n$.\\

Our starting point is the following conjecture of Gy\'{a}rf\'{a}s.

\begin{conjecture}\label{connConj}(\cite{GyarfasSurvey1},~\cite{GyarfasSurvey2}) Let $k$ be fixed. Given any colouring of the edges of $K_n$ in $k$ colours, we can find sets $A_1, A_2, \dots, A_{k-1}$ whose union is $[n]$, and colours $c_1, c_2, \dots, c_{k-1}$ such that $K_n[A_i, c_i]$ is connected for each $i \in [k-1]$.\end{conjecture}

This is an important special case of the well-known Lov\'{a}sz-Ryser conjecture, which we now state.

\begin{conjecture}\label{lovrys}(Lov\'{a}sz-Ryser conjecture.~\cite{Lov},~\cite{Ryser}) Let $G$ be a graph, whose maximum independent set has size $\alpha(G)$. Then, whenever $E(G)$ is $k$-coloured, we can cover $G$ by at most $(k-1)\alpha(G)$ monochromatic components.\end{conjecture}

Conjectures~\ref{connConj} and~\ref{lovrys} have attracted a great deal of attention. When it comes to the Lov\'{a}sz-Ryser conjecture, we should note the result of Aharoni (\cite{Aharoni}), who proved the case of $k = 3$. For $k \geq 4$, the conjecture is still open. The special case of complete graphs was proved by Gy\'{a}rf\'{a}s (\cite{GyarfasConn}) for $k \leq 4$, and by Tuza (\cite{TuzaConn}) for $k=5$. For $k > 5$, the conjecture is open.\\
Let us also mention some results similar in the spirit to Conjecture~\ref{bddConj}. In~\cite{Ruszinko}, inspired by questions of Gy\'arf\'as (\cite{GyarfasSurvey1}), Ruszink\'{o} showed that every $k$-colouring of edges of $K_n$ has a monochromatic component of order at least $n/(k-1)$ and of diameter at most 5. This was improved by Letzter (\cite{Letzter}), who showed that in fact there are monochromatic triple stars of order at least $n/(k-1)$. For more results and questions along these lines, we refer the reader to surveys of Gy\'arf\'as (\cite{GyarfasSurvey1},~\cite{GyarfasSurvey2}).\\

In a completely different direction, relating to contaction mappings on metric spaces, the following theorem is proved in~\cite{Commuting}. (We mention in passing that the current paper is self-contained, and in particular no knowledge of~\cite{Commuting} is assumed.)

\begin{theorem}\label{contrThm} There is an absolute constant $C > 0$ such that the following holds. If $0 < \lambda < C$, and if $\{f,g,h\}$ are commuting continuous maps on a complete metric space $(X,d)$ with the property that for any two distinct points $x,y \in X$ we have $\min \{d(f(x), f(y)), d(g(x), g(y)), d(h(x), h(y))\} \leq \lambda d(x,y)$, then the maps $f,g,h$ have a common fixed point. In fact, we may take $C = 10^{-23}$.\end{theorem}

Some of the ingredients in the proof of Theorem~\ref{contrThm} were the following simple lemmas. Note that Lemma~\ref{2cols} is in fact a classical observation due to Erd\H{o}s and Rado.

\begin{lemma}\label{2cols} Suppose that the edges of $K_n$ are coloured in two colours. Then we may find a colour $c$ such that $K_n[c]$ is connected and of diameter at most 3.\end{lemma}

\begin{lemma}\label{3cols} Suppose that the edges of $K_n$ are coloured in three colours. Then we may find colours $c_1, c_2$, (not necessarily distinct), and sets $A_1, A_2$ such that $A_1 \cup A_2 = [n]$, with $K_n[A_1, c_1], K_n[A_2, c_2]$ are each connected and of diameter at most 8.\end{lemma}

In~\cite{Commuting}, a common generalization of these statements and a strengthening of Conjecture~\ref{connConj} was conjectured.

\begin{conjecture}\label{bddConj} For every $k$, there is an absolute contant $C_k$ such that the following holds. Given any colouring of the edges of $K_n$ in $k$ colours, we can find sets $A_1, A_2, \dots, A_{k-1}$ whose union is $[n]$, and colours $c_1, c_2, \dots, c_{k-1}$ such that $K_n[A_i, c_i]$ is connected and of diameter at most $C_k$, for each $i \in [k-1]$.\end{conjecture}


The main result of this paper is

\begin{theorem}\label{mainthmbdd}Conjecture~\ref{bddConj} holds for 4 colours, and one may take $C_4 = 160$. \end{theorem} 

\subsection{An outline of the proof}

We begin the proof by establishing the weaker Conjecture~\ref{connConj} for the case of 4 colours. Although this was proved by Gy\'arf\'as in~\cite{GyarfasConn}, the reasons for giving a proof here are twofold. Firstly, we actually give a different reformulation of Conjecture~\ref{connConj} that has a more geometric flavour. The proof given here and the reformulation we consider emphasize the importance of the graph $G_k$, defined as a product of $k$ copies of $K_n$, to Conjecture~\ref{connConj}. Another reason for giving this proof is to make the paper self-contained.\\
We also need some auxiliary results about colourings with 2 or 3 colours, like Lemmas~\ref{2cols} and~\ref{3cols} mentioned above. In particular, we generalize the case of 2 colours to complete multipartite graphs. Another auxiliary result we use is the fact that $G_k$ essentially cannot have large very sparse graphs.\\

The main tool in our proof is the notion of \emph{$c_3, c_4$-layer mappings}, where $c_3, c_4$ are two colours. For $P \subset \mathbb{N}_0^2$, this is a mapping $L\colon P \to \mathcal{P}(n)$, (where $[n]$ is the vertex set of our graph), with the property that
\begin{enumerate}
\item sets $L(A)$ partition $[n]$ as $A$ ranges over $P$, 
\item and for $A, B \in P$ with $|A_1 - B_1|, |A_2 - B_2| \geq 2$, we have all edges between $L(A)$ and $L(B)$ coloured using only $c_3, c_4$. 
\end{enumerate}
This is a generalization of the idea that if we fix a vertex $x_0$ and we assign $A^{(x)} = (d_{c_1}(x_0, x), x_{c_2}(x_0, x)) \in \mathbb{N}_0^2$ to each vertex $x$, where $d_{c_1}, d_{c_2}$ are distances in colours $c_1, c_2$ (which are the remaining two colours), then if $A^{(x)}, A^{(y)}$ satisfy $|A^{(x)}_1 - A^{(y)}_1|, |A^{(x)}_2 - A^{(y)}_2| \geq 2$, the edge $xy$ cannot be coloured by $c_1$ or $c_2$.\\
Given a subset $P'$ of the domain $P$, we say that it is \emph{$k$-distant} if for all distinct $A, B \in P'$ we have $|A_1 - B_1|, |A_2 - B_2| \geq k$. Once we have all this terminology set up, we begin building up structure in our graph, essentially as follows:
\begin{itemize}
\item[] \textbf{Step 1.} We prove that if a $c_3, c_4$-layer mapping has a 3-distant set of size at least 4, then Theorem~\ref{mainthmbdd} holds.
\item[] \textbf{Step 2.} We continue the analysis of distant sets, and prove essentially that if a $c_3, c_4$-layer mapping has a 6-distant set of size at least 3, then Theorem~\ref{mainthmbdd} holds.
\item[] \textbf{Step 3.} We prove Theorem~\ref{mainthmbdd} when every colour induces a connected subgraph.
\item[] \textbf{Step 4.} We prove Theorem~\ref{mainthmbdd} when any two monochromatic components of different colours intersect.
\item[] \textbf{Step 5.} We put everything together to finish the proof.
\end{itemize}

\textbf{Organization of the paper.} In the next subsection, we briefly discuss a reformulation of Conjecture~\ref{connConj}. In Section 2, we collect some auxiliary results, including results on 2-colourings of edges of complete multipartite graphs and the results on sparse subgraphs of $G_k$ and indepenent sets in $G_3$. In Section 3, we prove Conjecture~\ref{connConj} for 4 colours, reproving a result of Gy\'arf\'as. The proof of Theorem~\ref{mainthmbdd} is given in Section 4, with subsections spliting the proof into the steps described above. Finally, we end the paper with some concluding remarks in Section 5.

\subsection{Another version of Conjecture~\ref{connConj}}

Let $l$ be an integer, define the graph $G_l$ with vertex set $\mathbb{N}_0 ^l$ and put an edge between any two sequences that differ at every coordinate. Equivalently, $G_l$ is the direct product of $l$ copies of $K_{\mathbb{N}_0}$ (the complete graph on the vertex set $\mathbb{N}_0$). We formulate the following conjecture.\\[6pt]

\begin{conjecture}\label{altConj} Given a set finite set of vertices of $X \subset \mathbb{N}_0^l$, we can find $l$ sets $X_1, \dots, X_l \subseteq X$ that cover $X$ and each $X_i$ is either contained in a hyperplane of the form $\{x_i = c\}$ or $G_l[X_i]$ is connected.\end{conjecture}

This conjecture is actually equivalent to Conjecture~\ref{connConj}.

\begin{proposition} Conjectures~\ref{connConj} and~\ref{altConj} are equivalent for $k = l+1$. \end{proposition}

\begin{proof} \emph{Conjecture~\ref{connConj} implies Conjecture~\ref{altConj}.} Let $X \subset \mathbb{N}_0^l$ be a finite set. Let $n = |X|$ and define an $(l+1)$-colouring $\chi\colon E(K_n) \to [l+1]$ by setting $\chi(xy) = i$, where $i$ is the smallest coordinate index such that $x_i = y_i$, otherwise, when $x$ and $y$ differ in all coordinates, set $\chi(xy) = l+1$. If Conjecture~\ref{connConj} holds, we may find sets $A_1, A_2, \dots, A_l$ that cover $[n]$, and colours $c_1, c_2, \dots, c_l$  such that $K_n[A_i, c_i]$ are all connected. Fix now any $i$, and let $B \subset X$ be the set of vertices corresponding to $A_i$. If $c_i \leq l$, then for any $x,y \in B$, there is a sequence of vertices $z_1, z_2, \dots, z_m \in B$ such that $x_i = (z_1)_i = (z_2)_i = \dots = (z_m)_i = y_i$, so $x_i = y_i$. Hence, $B$ is subset of the plane $\{x_i = v\}$ for some value $v$. Otherwise, if $c = l+1$, that means that the edges of $K_n[A_i, c_i]$ correspond to edges of $G[B]$, so $G[B]$ is connected, as desired.\\
\indent\emph{Conjecture~\ref{altConj} implies Conjecture~\ref{connConj}.} Let $\chi\colon E(K_n) \to [k]$ be any $k$-colouring of the edges of $K_n$. For every colour $c$, look at components $C^{(c)}_1, \dots, C^{(c)}_{n_c}$ of $K_n[c]$. For each choice of $x_1, x_2, \dots, x_{k-1}$ with $x_c \in [n_c]$ for $c \in [k-1]$, we define $C_x = C_{x_1, x_2, \dots, x_{k-1}} = \cap_{c \in [k-1]} C^{(c)}_{x_c}$, which is the intersection of monochromatic components, one for each colour except $k$. Let $X \subset \mathbb{N}^{k-1}$ be the set of all $(k-1)$-tuples $x$ for which $C_x$ is non-empty. If Conjecture~\ref{altConj} holds, then we can find $A_1, A_2, \dots, A_{k-1}$ that cover $X$ such that each $A_i$ is either contained in a hyperplane, or induces a connected subgraph of $G_{k-1}$. If $A_i \subset \{x_c = v\}$, then the corresponding intersections $C_x$ for $x \in A_i$ are all subset of $C^{(c)}_v$. On the other hand, if $G_{k-1}[A_i]$ is connected, then taking any adjacent $x,y \in G_{k-1}[A_i]$, we have that $x_c \not= y_c$ for all $c \in [k-1]$. Hence all the edges of between $C_x$ and $C_y$ are coloured by $k$. Hence, all the sets $C_x$ for $x \in A_i$ are subset of the same component of $K_n[k]$. This completes the proof of the proposition.\end{proof}

\section{Auxiliary results}

As suggested by its title, this section is devoted to deriving some auxiliary results. Firstly we extend Lemma~\ref{2cols} to complete multipartite graphs. The case of bipartite graphs is slightly different from the general case of more than 2 parts, and is stated separately. We also introduce additional notation. Given a colour $c$ and vertices $x,y$ we write $d_c(x,y)$ for the distance between $x$ and $y$ in $G[c]$. If they are not in the same $c$-component, we write $d_c(x,y) = \infty$. In particular, $d_c(x,y) < \infty$ means that $x, y$ are in the same component of $G[c]$. Further, we write $B_c(x, r)$ for the \emph{$c$-ball of radius $r$ around $x$}, defined as $B_c(x, r) = \{y\colon d_c(x,y) \leq r\}$, where $c$ is a colour, $x$ is a vertex, and $r$ is a nonnegative integer. For any graph $G$, throughout the paper, the \emph{diameter} of $G$, written $\diam G$, is the supremum of all finite distances between two vertices of $G$. Thus, $\diam G = \infty$ only happens when $G$ has arbitrarily long induced paths (as we focus on the finite graphs in this paper, this will not occur). For a colour $c$ and a set of vertices $A$, the \emph{$c$-diameter} of $A$, writen $\diam_c A$, is the diameter of $G[A, c]$. We use the standard notation for complete multipartite graphs, so $K_{n_1, n_2, \dots, n_r}$ stands for the graph with $r$ vertex classes, of sizes $n_1, n_2 \dots, n_r$, and all edges between different classes are present in the graph.

\begin{lemma}\label{2colsBip} Suppose that the edges of $G = K_{n_1, n_2}$ are coloured in two colours. Then, one of the following holds:
\begin{enumerate}
\item either there is a colour $c$, such that $G[c]$ is connected and of diameter at most 10, or
\item there are partitions $[n_1] = A_1 \cup B_1$ and $[n_2] = A_2 \cup B_2$ such that all edges in $A_1 \times A_2 \cup B_1 \times B_2$ are of one colour, and all the edges in $A_1 \times B_2 \cup B_1 \times A_2$ are of the other colour.
\end{enumerate}
\end{lemma}

\begin{proof} Let $\chi$ be the given colouring. We start by observing the following. If there are two vertices $v_1, v_2$ such that for colour $c_1$ the inequality $6 \leq d_{c_1}(v_1,v_2) < \infty$ holds, then for every vertex $u$ such that $\chi(uv_1) = c_1$, we must also have $d_{c_2}(u,v_1) \leq 3$, where $c_2 \not= c_1$ is the other colour. Indeed, let $v_1 = w_0, w_1, w_2, \dots, w_r = v_2$ be a minimal $c_1$-path from $v_1$ to $v_2$. Hence $r \geq 6$, the vertices $w_i$ with the same parity of index belong to the same vertex class of $G = K_{n_1, n_2}$ and the edges $v_1w_3 = w_0w_3, w_3w_6, w_6u \in E(G)$ are all of colour $c_2$ (otherwise, we get a contradiction to the fact that $d_{c_1}(w_i, v_2) = r - i$), implying that $d_{c_2}(v_1, u) \leq 3$.\\
Now, suppose that a $c_1$-component $C_1$ has diameter at least 7. The observation above tells us that if a vertex $y$ is adjacent to $x_1$, and $d_{c_2}(x_1,y) > 1$, then $\chi(x_1, y) = c_1$, so $d_{c_2}(x_1, y) \leq 3$. Hence, every vertex $y$ adjacent to $x_1$ in $G$, satisfies $d_{c_2}(x_1, y) \leq 3$. Similarly, any vertex $y$ adjacent to $x_2$ satisfies $d_{c_2}(x_2, y) \leq 3$. But, $x_1, x_2$ are in different vertex classes (as their $c_1$-distance is odd), so their neighbourhhoods cover the whole vertex set, and $x_1 x_2$ is an edge as well, from which we conclude that $G[c_2]$ is connected and of diameter at most 9. Thus, if any monochromatic component has diameter at least 7, the lemma follows, so assume that this does not occur.\\
Now we need to understand the monochromatic components. From the work above, it suffices to find monochromatic components of the desired structure, the diameter is automatically bounded by 6. Suppose that there are at least 3 $c_1$-components, $X_1 \cup X_2, Y_1 \cup Y_2, Z_1 \cup Z_2$ with $X_1, Y_1, Z_1$ subsets of one class of $K_{n_1, n_2}$ and $X_2, Y_2, Z_2$ subsets of the other. Let $u, v \in X_1\cup Y_1\cup Z_1$ be arbitrary vertices. Then we can find $w \in X_2\cup Y_2\cup Z_2$ in different $c_1$-component from $u, v$. Hence, $\chi(uw) = \chi(wv) = c_2$, so $d_{c_2}(u,v) \leq 2$. Therefore, both vertex classes of $G$ are $c_2$-connected and consequently the whole graph is $c_2$-connected.\\
Finally, assume that each colour has exactly 2 monochromatic components. Let $[n_1] = A_1 \cup B_1, [n_2] = A_2 \cup B_2$ be such that $A_1 \cup A_2, B_1 \cup B_2$ are the $c_1$-components. Hence, $A_1 \cap B_1 = A_2 \cap B_2 = \emptyset$, and all edges in $A_1 \times B_2$ and $B_1 \times A_2$ are of colour $c_2$. Thus, sets $A_1 \cup B_2$ and $B_1 \cup A_2$ are $c_2$-connected and cover the vertices of $G$, so they must be the 2 $c_2$-components. Thus, all edges in $A_1\times A_2$ and $B_1 \times B_2$ must be coloured by $c_1$, proving the lemma.
\end{proof}

\begin{lemma} \label{mult2col}Let $r\geq 3$, and suppose that $G = K_{n_1, n_2, \dots, n_r}$ is a complete $r$-partite graph . Suppose that the edges of $G$ are 2-coloured. Then, there is a colour $c$ such that $G[c]$ is connected and of diameter at most $C_r$, where we can take $C_3 = 20$, and $C_r = 60$ for $r > 3$.\end{lemma}

\begin{proof} \textbf{Assume first that $r=3$.} Let $A, B, C$ be the vertex classes. We shall use Lemma~\ref{2colsBip} throughout this part of the proof, applying to every pair of vertex classes. We distinguish three cases, motivated by the possible outcomes of Lemma~\ref{2colsBip} (although not exactly these outcomes, but resembling them).\\
\textbf{Observation.} Suppose that $D, E, F$ is a permutation of $A, B, C$ and that $D \cup E$ is contained in a $c_1$-component of diameter at most $N_1$, and $D \cup F$ for each colour splits into two monochromatic components, all of diameter at most $N_2$. Then, $G[c_1]$ is connected and of diameter at most $N_1 + 2N_2$.\\
\textbf{Case 1.} Suppose that $D, E, F$ is a permutation of $A, B, C$, and that Lemma~\ref{2colsBip} gives different outcomes when applied to pairs $D, E$ and $D, F$. Then, by the Observation, there is a colour $c$ such that $G[c]$ is connected and of diameter at most 14. (We took $N_1 = 10$ and $N_2 = 2$.)\\
\textbf{Case 2.} Suppose that $D, E, F$ is a permutation of $A, B, C$, and that Lemma~\ref{2colsBip} gives a single monochromatic component for each of pairs $D, E$ and $D, F$. If we use the same colour $c$ for both pairs, then $G[c]$ is connected and of diameter at most 20. Otherwise, let $D\cup E$ be $c_1$-connected, and let $D \cup F$ be $c_2$-connected, with $c_1 \not= c_2$. Apply Lemma~\ref{2colsBip} to $E, F$. If it results in a single monochromatic component, it must be of colour $c_1$ or $c_2$, so once again $G[c]$ has diameter at most 20 for some $c$. Finally, if $E \cup F$ splits in two pairs of monochromatic components, by Observation $G[c]$ has diameter at most 14, for some $c$.\\
\textbf{Case 3.} Lemma~\ref{2colsBip} gives the second outcome for each pair of vertex classes. Look at complete bipartite graphs $G[A\cup B]$ and $G[A\cup C]$. Then, we have partitions $A = A_1 \cup A_2 = A'_1 \cup A'_2$, $B = B_1 \cup B_2$ and $C = C_1 \cup C_2$ such that all edges $(A_1 \times B_1)\cup(A_2 \times B_2) \cup (A'_1 \times C_1) \cup (A'_2 \times C_2)$ receive colour $c_1$, while the edges $(A_1 \times B_2)\cup(A_2 \times B_1) \cup (A'_1 \times C_2) \cup (A'_2 \times C_1)$ take the other colour $c_2$. If $\{A_1, A_2\} \not= \{A'_1, A'_2\}$, then we must have that some $A_i$ intersects both $A'_1, A'_2$, or vice-versa. In particular, since any two vertices $x, y$ in the same set among $A_1, A_2, A'_1, A'_2$ obey $d_{c_1}(x,y) \leq 2$, this means that for any two vertices $x,y \in A$, we have $d_{c_1}(x,y) \leq 6$. Now, every point in $B \cup C$ in on $c_1$-distance at most 1 from a vertex in $A$, so $G[c_1]$ is connected and of diameter at most 8. Hence, we may assume that $A_1 \cup A_2$ and $A'_1 \cup A'_2$ are the same partitions of $A$, and similarly for $B$ and $C$, we get the same partition for both pairs of vertex classes involving each of $B$ and $C$. Let $A = A_1 \cup A_2, B = B_1 \cup B_2, C = C_1 \cup C_2$ be these partitions, so the colouring is constant on each product $A_i \times B_j, A_i \times C_j, B_i \times C_j$, $i,j \in \{1,2\}$. Renaming $B_i, C_j$, we may also assume that $A_1 \times B_1, A_2 \times B_2, A_1 \times C_1, A_2 \times C_2$ all receive colour $c_1$. Thus $A_1 \times B_2, A_2 \times B_1, A_1\times C_2, A_2 \times C_1$ all receice colour $c_2$. But looking at the colour $c$ of $B_1 \times C_2$, we see that $G[c]$ is connected and of diameter at most 5. This finishes the proof of the case $r=3$, and we may take $C_3 = 20$.\\
 
\textbf{Now suppose that $r > 3$.} Let $V_1, V_2, \dots, V_r$ be the vertex classes. Fix the vertex class $V_r$, and look at the 2-colouring $\chi'$ of the edges of $K_{r-1}$ defined as follows: whenever $i, j \in [r-1]$ are distinct, then applying the case $r=3$ of this lemma that we have just proved to the subgraph induced by $V_i \cup V_j \cup V_r$, we get a colour $c$ such that $G[V_i \cup V_j \cup V_r, c]$ has diameter at most 20; we set $\chi'(ij) = c$. By Lemma~\ref{2cols}, we have a colour $c$ such that $K_{r-1}[c]$ is of diameter at most 3 for the colouring $\chi'$. Returning to our original graph, we claim that $G[c]$ has diameter at most 60. Suppose that $x,y$ are any two vertices of $G$. If any of these points lies in $V_r$, or if they lie in the same $V_i$, then we can pick $i, j$ such that $y \in V_i \cup V_j \cup V_r$ and $\chi'(ij) = c$. Hence, by the definition of $\chi'$, we actually have $d_c(x,y) \leq 20$ in $G$. Now, assume that $x,y$ lie in different vertex classes and outside of $V_r$. Let $x \in V_i, y \in V_j$. Under the colouring $\chi'$ of $K_{r-1}$ we have that $d_c(ij) \leq 3$, so we have a sequence $i_1 = i, i_2, \dots, i_s = j$, with $s \leq 4$, such that $\chi'(i_1 i_2) = \dots = \chi'(i_{s-1} i_s) = c$. For each $t$ between 1 and $s$, pick a representative $x_t \in V_{i_t}$, with $x = x_1, y= x_s$. Then, $d_c(x_{t-1}, x_t) \leq 20$, so $d_c(x,y) = d_c(x_1, x_s) \leq 60$, as desired.\end{proof}

\subsection{Induced subgraphs of $G_l$}

Recall that $G_l$ is the graph on $\mathbb{N}^l$, with edges between pairs of points whose all coordinates differ. In this subsection we prove a few properties of such graphs, particularly focusing on $G_3$. We begin with a general statement, which will be reproved for specific cases with stronger conclusions.

\begin{lemma}\label{genSparse} If $S$ is a set of vertices in $G_l$ and the maximal degree of $G[S]$ is at most $d$, then the number of non-isolated vertices of $G[S]$ is at most $O_{l,d}(1)$.\end{lemma}

\begin{proof} By Ramsey's theorem we have an $N$ such that whenever $E(K_N)$ is coloured using $2^l - 1$ colours, there is a monochromatic $K_{l+1}$. Let $S'$ be the set of non-isolated vertices in $S$. We show that $|S'| < (d^2 + d +1)N$. Suppose contrary, since the maximal degree is at most $d$, we have a subset $S'' \subset S$ of size $|S''| \geq N$ such that sets $s \cup N(s)$ are disjoint for all $s \in S''$ (simply pick a maximal such subset, their second neighbourhoods must cover the whole $S'$). In particular, $S''$ is an independent set in $G_l$, so for every pair of vertices $x, y \in S$, the set $I(x,y) = \{i \in [l]\colon x_i = y_i\}$ is non-empty. Thus, $I\colon E(K_{S''}) \to \mathcal{P}(l)\setminus\{\emptyset\}$ is $2^l-1$ colouring of the edges of a complete graph $K_{S''}$ on the vertex set $S''$. By Ramsey's theorem, there is a monochromatic clique on subset $T \subset S''$ of size at least $l+1$, whose edges are coloured by some set $I_0 \not=\emptyset$. Take a vertex $t \in T$, and since $t$ is not isolated and the neighbourhoods of vertices in $S''$ are disjoint, we can find $x \in S'$ such that $tx$ is an edge, but $t'x$ is not for other $t' \in T$. Hence, $x_i \not= t_i$ for all $i\in[l]$ and for distinct $t', t'' \in T$ we have $t'_i = t''_i$ if and only if $i \in I_0$. Thus, $x_i \not= t'_i$ for all $t' \in T$ and $i \in I_0$. But, $x t'$ is not an edge for $t' \in T \setminus \{t\}$, so we always have $i \in [l] \setminus I_0$ such that $x_i = t'_i$. But, for each $i \in I_0$, the values of $t'_i$ are distinct for each $t' \in T$. Hence, for each $i$, there is at most one vertex $t' \in T\setminus \{t\}$ such that $x_i = t'_i$. Therefore $|T| - 1 \leq |[l] \setminus I_0| \leq l-1$, so $|T| \leq l$, which is a contradiction.\end{proof}

We may somewhat improve on the bound in the proof of the lemma above by observing that for colour $I_0$ we only need a clique of size $l - |I_0| + 2$. Thus, instead of Ramsey number
$$R(\underbrace{l+1, l+1, \dots, l+1}_{2^l-1}),$$
we could use
$$R(l+2 - |I_1|, l+2 - |I_2|, \dots, l+2 - |I_{2^l-1}|),$$
where $I_{i}$ are the non-empty sets of $[l]$. But, even for paths in $G_3$, which we shall use later, taking $l = 3, d = 2$, we get the final bound of $7 R(2, 3, 3, 3, 4, 4, 4)$, where 7 comes from $d^2 + d + 1$ factor we lose when moving from $S'$ to $S''$. We now improve this bound.

\begin{lemma}\label{g3path} If $S$ is a set of vertices of $G_3$ such that $G_3[S]$ is a path, then $|S| \leq 30$.\end{lemma}

\begin{proof} Let $S = \{s_1, s_2, \dots, s_r\}$ be such that $s_1, s_2, \dots, s_r$ is an induced path in $G_3$, so the only edges are $s_i s_{i+1}$.\\
\textbf{Case 1.} For all $i \in \{4, 5, \dots, 10\}$, $s_i$ coincides with one of $s_1$ or $s_2$ in at least two coordinates.\\
\indent Since $s_1 s_2$ is an edge, $s_1$ and $s_2$ have all three coordinates different. Thus, for $i \in \{4,5,\dots, 10\}$, we have $(s_i)_c \in \{(s_1)_c, (s_2)_c\}$ for all coordinates $c$. Hence, there are only at most 6 possible choices of $s_i$ (as $s_i \not= s_1, s_2$), so $r \leq 9$.\\
\textbf{Case 2.} There is $i_0 \in \{4,5, \dots, 10\}$ with at most one common coordinate with each of $s_1, s_2$. Since $s_1 s_{i_0}, s_2 s_{i_0}$ are not edges, w.l.o.g. we have $s_1 = (x_1, x_2, x_3), s_2 = (y_1, y_2, y_3), s_{i_0} = (x_1, y_2, z_3)$, where $x_i \not= y_i$, $z_3 \notin \{x_3, y_3\}$. Consider any point $s_j$, for $j\geq i_0 + 2$. It is not adjacent to any of $s_1, s_2, s_{i_0}$. If $(s_j)_1 = x_1$ and $(s_j)_2 \not= y_2$, then $(s_j)_3 = y_3$. Similarly, if $(s_j)_1 \not= x_1$ and $(s_j)_2 = y_2$, then $(s_j)_3 = x_3$. Also, if $(s_j)_1 \not= x_1, (s_j)_2 \not= y_2$, then $s_j = (y_1, x_2, z_3)$. Hence, for $j\geq i_0 + 2$, the point $s_j$ is on one of the lines
$$(x_1, y_2, \cdot), (x_1, \cdot, y_3), (\cdot, y_2, x_3) \text{ or it is the point }(y_1, x_2, z_3),$$
where $(a, b, \cdot)$ stands for the line $\{(a,b,z)\colon z \text{ arbitrary}\}$, etc. Note that a point on $(x_1, y_2, \cdot)$ is not adjacent to any point on $(\cdot, y_2, x_3)$, and the same holds for lines $(x_1, y_2, \cdot)$ and $(x_1, \cdot, y_3)$. Hence, along out path, a point on the line $(x_1, \cdot, y_3)$ is followed either by a point on $(\cdot, y_2, x_3)$ or the point $(y_1, x_2, z_3)$ (the latter may happen only once). In any case, if $|S| \geq 30$, then among $s_{i_0 + 2}, s_{i_0 + 3}, \dots, s_{i_0 + 20}$, we must get a contiguous sequence $s_{j}, s_{j+1}, \dots, s_{j+7}$ of points
$$s_{j}, s_{j+2}, s_{j+4}, s_{j+6} \in (x_1, \cdot, y_3), s_{j+1},s_{j+3}, s_{j+5}, s_{j+7}  \in (\cdot, y_2, x_3).$$
Finally, we look at $A = s_j, B = s_{j+2}, C = s_{j+5}, D = s_{j+7}$. These four points form an independent set, but $A \not= B$ gives $A_2 \not= B_2$, so one of $A_2 \not= y_2, B_2 \not= y_2$ holds, and similarly, one of $C_1 \not= x_1, D_1 \not= x_1$ holds as well. Choosing a point among $A, B$ and a point among $C, D$ for which equality does not hold gives an edge, which is impossible. \end{proof}

Finally, we study independent sets in $G_3$. Note that Lemma~\ref{genSparse} in this case does not tell us anything about the structure of such sets. When we refer to line or planes, we always think of very specific cases, namely the lines are the sets of the form $\{x\colon x_i = a, x_j = b\}$ and the planes are $\{x\colon x_i = a\}$. Similarly, collinearity and coplanarity of points have stronger meaining, and imply that points lie on a common line or plane defined as above. 

\begin{lemma} Let $S$ be a set of vertices in $G_3$. If every two points of $S$ are collinear, then $S$ is a subset of a line. If every three points of $S$ are coplanar, then $S$ is a subset of a plane. \end{lemma}

\begin{proof} We first deal with the collinear case. Take any pair of points, $x, y \in S$, w.l.o.g. they coincide in the first two coordinates. Take third point $z \in S$. If $z$ does not share the values of the first 2 coordinates with $x$ and $y$, then we must have $x_3 = z_3 = y_3$, which is impossible. As $z$ was arbitary, we are done.\\
Suppose now that we have all triples coplanar. W.l.o.g. we have a noncolinear pair $x, y$, which only coincide in the first coordinate. Then all other points may only be in the plane $\{p\colon p_1 = x_1\}$.\end{proof}

\begin{lemma}\label{4struct}(Structure of the independent sets of size 4.) Given an independent set $I$ of $G_3$ of size 4 (at least) one of the following alternatives holds
\begin{itemize}
\item[] (\textbf{S1}) $I$ is coplanar, or
\item[] (\textbf{S2}) $I = \{(a,b,c), (a',b',c), (a',b,c'), (a,b',c')\}$, where $a \not= a'; b \not= b'$ and $c \not= c'$, or
\item[] (\textbf{S3}) up to permutation of coordinates $I = \{(a,b,c), (a,b,c'), (a,b',x), (a',b,x)\}$, where $a \not= a'; b \not= b'$ and $c \not= c'$.
\end{itemize}
 \end{lemma}

\begin{proof} Suppose that $I = \{A, B, C, D\}$ is not a subset of any plane. We distinguish between two cases.\\

\noindent\textbf{Case 1.} There are no collinear pairs in $I$.\\
Let $A = (a, b, c)$. But $AB$ is not an edge and not colinear so $A$ and $B$ differ in precisely two coordinates. Thus, w.l.o.g. $B = (a', b', c)$ where $a \not= a'$ and $b \not= b'$. If $C_3$ also equals $c$, then we must have $C_3 = (a'', b'', c)$ with $a''$ different from $a,a'$ and $b''$ from $b, b'$. However, looking at $D$, we cannot have $D_3 = c$ as otherwise $I \subset \{x_3 = c\}$, so $D$ must differ at all three coordinates from one of the points $A, B, C$, making them joined by an edge, which is impossible. Thus $C_3 = c'$, with $c' \not= c$. Since $AC$ and $BC$ are not edges, $C \in \{(a, b', c'), (a', b, c')\}$. The same argument works for $D$, so $D_3 = c'' \not= c$, and $D \in \{(a, b', c''), (a', b, c'')\}$. However, if $c' \not= c''$, then $C, D$ are either collinear or adjacent in $G_3$, which are both impossible. Hence $c'' = c'$, and $\{C, D\} = \{(a,b',c'), (a',b,c')\}$, as desired.\\

\noindent\textbf{Case 2.} W.l.o.g. $A$ and $B$ are collinear.\\
Let $A = (a,b, c), B = (a,b,c')$ with $c \not= c'$. Since $\{x_1 = a\}$ does not contain the whole set $I$, we have w.l.o.g. $C_1 = a' \not = a$. If $C_2 \not= b$, then $AC$ or $BC$ is an edge, which is impossible. Therefore, $C_2 = b$. Hence $D_2 = b' \not= b$, and by similar argument $D_1 = a$. Finally $CD$ is not an edge, so their third coordinate must be the same, proving the lemma.\end{proof}

\begin{lemma}(Structure of the independent sets of size 5.) Given an independent set $I$ of $G_3$ of size 5 (at least) one of the following alternatives holds
\begin{enumerate}
\item $I$ is coplanar, or
\item $I$ is a subset of a union of three lines, all sharing the same point.
\end{enumerate} \end{lemma}

\begin{proof} List the vertices of $I$ as $x_1, x_2, x_3, x_4, x_5$. W.l.o.g. $x_1, x_2, x_3$ are not coplanar. By the previous lemma, $\{x_1, x_2, x_3, x_i\}$  for $i=4,5$ may have structure \textbf{S2} or \textbf{S3}. But if both structures are \textbf{S2}, then we must have that in both quadruples, at each coordinate, each value appears precisely two times. This implies $x_4 = x_5$. Hence, w.l.o.g. $\{x_1, x_2, x_3, x_4\}$ has structure \textbf{S3}. Therefore, assume w.l.o.g. that
$$x_1 = (1,0,0), x_2 = (0,1,0), x_3 = (0,0,1), x_4 = (0, 0, c')$$
for some $c' \not= 1$ (which corresponds to the choice $a = 0, a' = 1, b = 0, b' = 1, x = 0, c = 1$ in the previous Lemma, switching the roles of  $c$ and $c'$ if necessary). Looking at $\{x_1, x_2, x_3, x_5\}$, if it had \textbf{S2} for its structure, we would get $x_5 = (1,1,1)$, which is adjacent to $x_4$, and thus impossible. Hence $\{x_1, x_2, x_3, x_5\}$ also has structure \textbf{S3}. Permutting the coordinates only permutes $x_1, x_2, x_3$, and does not change the number of zeros in $x_5$. Thus, w.l.o.g.
$$\{(1,0,0), (0,1,0), (0,0,1), x_5\} = \{x_1, x_2, x_3, x_5\} = \{(d,e,f), (d,e,f'), (d',e,y), (d,e',y)\},$$
for some $d \not= d', e \not= e', f \not= f'$. But in the first coordinate, only zero can appear three times, so $d = 0$. Similarly, $e = 0$, so $x_5 \in (0,0, \cdot)$, after a permutation of coordinates. Thus $x_5$ has at least 2 zeros, so our independent set $I$ is a subset of the union of lines passing through the point $(0,0,0)$, as required.\end{proof}

\section{Conjecture~\ref{connConj} for 4 colours}

In this short section we reprove the result of Gy\'arf\'as.

\begin{theorem} (Gy\'arf\'as) Conjecture~\ref{connConj} for 4 colours and Conjecture~\ref{altConj} for $G_3$ are true.\label{connThm}\end{theorem}

\begin{proof} By the equivalence of conjectures, it suffices to prove Conjecture~\ref{altConj} for $G_3$. Let $X$ be the given finite set of vertices in $G_3$. Assume that $G_3[X]$ has at least 4 components, otherwise we are done immediately. By a \emph{representatives set} we mean any set of vertices that contains at most one vertex from each component of $X$. A \emph{complete representative set} is a representative set that intersects every component of $X$.

\begin{observation} If there are three colinear points, each in different component, then $X$ can be covered by two planes. In particular, if two planes do not suffice, then among every three points in different components, there is a non-colinear pair.\end{observation}

\begin{proof} W.l.o.g. these are points $(0, 0, 1), (0,0,2), (0,0,3)$. Then, unless $X \subset \{x_1 = 0\} \cup \{x_2 = 0\}$, we have a point of the form $(a,b,c)$ with $a,b$ both non-zero, so it is a neighbour of at least two of the points we started with, contradicting the fact that they belong to different components. For the second part, recall that if every pair in a triple is colinear, then the whole triple lies on a line.\end{proof}

By the observation above, every representative set of size at least 3 has a noncollinear pair. Suppose firstly that every complete representative set is a subset of a plane. Pick a complete representative set $\{x_1, x_2, \dots, x_r\}$, with $x_i \in C_i$, where $C_i$ are the components. W.l.o.g. $x_1, x_2$ is a noncollinear pair, therefore, it determines a plane $\pi$, forcing components $C_3, C_4, \dots, C_r$ to be entirely contained in this plane. Hence, we may cover the whole set $X$ by components $C_1$ and $C_2$, and the plane $\pi$. Therefore, we may assume that we have a representative set of size three which does not lie in any plane.\\

\noindent\textbf{Case 1.} $X$ has more than 4 components.\\
\indent Let $x_1, x_2, x_3$ be a representative set, $x_i \in C_i$, which is not coplanar. Then, for any choice of $y_4, \dots, y_r$, such that $\{x_1, x_2, x_3, y_4, \dots, y_r\}$ is a complete representative set, we have 3 lines  that meet in a single point, that contain all these points. Observe that this structure is determined entirely by $x_1, x_2, x_3$. Indeed, since these three points are not coplanar, they cannot coincide in any coordinate. However, since there are at least 5 components, $x_1, x_2, x_3$ extend to an independent set of size 5, which must be a subset of three lines sharing a point $p$. But we can identify $p$, since $p_i$ must be the value that occurs precisely two times among $(x_1)_i, (x_2)_i, (x_3)_i$, and hence the lines are $l_1 = px_1, l_2 = px_2, l_3 = px_3$. Thus, the union of lines $l_1, l_2, l_3$ contains whole components $C_4, \dots, C_r$ and $x_i \in l_i$. By the Observation above, each $l_i$ has representatives from at most two components. Hence, we may not have the common point of the three lines $p$ present in  $X$, as otherwise some line $l_i$ would have three components meeting it. W.l.o.g. $l_2, l_3$ intersect two components, and $l_1$ may intersect 1 or 2. Then, picking any $y \in l_2$ in a different component than that of $x_2$ and any $z \in \l_3$ with a component different from that of $x_3$, using the argument above applied to $\{x_1, y, z\}$ instead of $\{x_1, x_2, x_3\}$, we deduce that $C_2 \subset l_2$, $C_3 \subset l_3$. Thus, we actually have singleton components $C_2, C_3, \dots, C_r$. Finally, any point in $C_1$ must be either in the plane of $l_2, l_3$ or on the line $l_1$, so we can cover by two planes.\\

\noindent\textbf{Case 2.} $X$ has precisely 4 components and there exists a coplanar complete representative set.\\
\indent Let $x_1, x_2, x_3, x_4$ be a complete representative set, with $x_i \in C_i$. W.l.o.g. we have $x_i = (a_i, b_i, 0)$. As a few times before, we do not have a collinear triple among these 4 points, so each of the sequences $(a_i)_{i=1}^4$ and $(b_i)_{i=1}^4$ has the property that a value may appear at most twice in the sequence.\\
Suppose for a moment that each of these two sequences has at most one value that appears twice. Write $v$ for the value that appears two times in $(a_i)$, if it existis, and let $v$ be the corresponding value for $(b_i)$. If we take a point $y$ outside the plane $(\cdot, \cdot, 0)$, then the number of appearances of $y_1$ in $(a_i)$ and $y_2$ in $(b_i)$ combined is at least three. So, either $y_1$ is the unique doubly-appearing value $u$ for $a_i$ or is $y_2 = v$, so the three planes $(u, \cdot, \cdot), (\cdot, v, \cdot)$ and $(\cdot, \cdot, 0)$ cover $X$.\\
Now, assume that w.l.o.g. has two doubly-appearing values, i.e. $a_1 = a_2 = u \not= a_3 = a_4 = v$. If $y$ is outside the plane $(\cdot, \cdot, 0)$, then if $y_1 \not= u$, one of the pairs $x_1y, x_2 y$ must be an edge, so $x_3y$ and $x_4y$ are not edges, so we must have $y_1 = v$. Similarly, if $y$ is outside the plane  $(\cdot, \cdot, 0)$ and $y_1 \not= v$, then $y_1 = u$. Hence, for all points $y \in X$, we have $y_1 \in \{u,v\}$ or $y_3 = 0$, and three planes cover once again.\\

\noindent\textbf{Case 3.} $X$ has precisely 4 components, but no complete representative set is coplanar.\\
\indent Thus, by Lemma~\ref{4struct}, every complete representative set has either \textbf{S2} or \textbf{S3} as its structure. Observe that if \textbf{S2} is always the structure, then all the components are singleton, and we are done by taking a plane to cover two vertices. So, there is a representative set with structure \textbf{S3}. Take such a representative set $x_1, x_2, a, b$, w.l.o.g. $x_1 = (1,0,0), x_2 = (2,0,0)$. Take any $y$ that shares the component with $a$, and any $z$ that shares the component with $b$. Then, $x_1, x_2, y, z$ is also a complete representative set, so it is not coplanar. But, as $x_1, x_2$ are collinear, it may not have structure \textbf{S2}, so the structure must be \textbf{S3}, which forces $y_1 = z_1$. Hence, we can cover $X$ by components of $x$ and $y$ and the plane $(a_1, \cdot, \cdot)$. This completes the proof.\end{proof}


Note that the theorem is sharp -- we can take $X = \{0, e_1, e_2, e_3, e_1 + e_2, e_1 + e_3, e_2 + e_3\}$, where $e_1 = (1,0,0), e_2 = (0,1,0), e_3 = (0,0,1)$.

\section{Conjecture~\ref{bddConj} for 4 colours}

Recall, by a \emph{diameter} of a colour $c$, written $\operatorname{diam}_c$, we mean the maximal distance between vertices sharing the same component of $G[c]$. In the remaining part of the paper, for a given 4-colouring $\chi\colon E(K_n) \to 4$, we say that $\chi$ \emph{satisfies Conjecture~\ref{bddConj} with (constant) $K$} if there are sets $A_1, A_2, A_3$ whose union is $[n]$ and colours $c_1, c_2, c_3$ such that each $K_n[A_i, c_i]$ is connected and of diameter at most $K$. Thus, our goal can be phrased as: there is an absolute constant $K$ such that every 4-colouring $\chi$ of $E(K_n)$ satisfies Conjecture~\ref{bddConj} with $K$.\\ 
We begin the proof of the main result by observing that essentially we may assume that at least two colours have arbitrarily large diameters. We argue by modifying the colouring slightly.

\begin{lemma}\label{3smalldiam} Suppose $\chi$ is a 4-colouring of $E(K_n)$ such that three colours have diameters bounded by $N_1$. Then $\chi$ satisfies Conjecture~\ref{bddConj} with $\max \{N_1, 30\}$. \end{lemma}

\begin{proof} Write $G = K_n$, and observe that if a point does not receive all 4 colours at its edges, we are immediately done. Let $\chi$ be the given colouring of the edges, and let colours 1, 2 and 3  have diameter bounded by $N_1$. We begin by modifying the colouring slightly. Let $xy$ be any edge coloured by colour 4. If $x$ and $y$ share the same component in $G[c]$ for some $c \in \{1,2,3\}$, change the colour of $xy$ to the colour $c$ (if there is more than one choice, pick any). Note that such a modification does not change the monochromatic components, except possibly shrinking the components for the colour 4. Let $\chi'$ stand for the modified colouring.\\
Observe that the diameter of colour 4 in $\chi'$ is also bounded. Begin by listing all the components for colours $i \in \{1,2,3\}$ as $C^{(i)}_1, C^{(i)}_2, C^{(i)}_3, \dots$. For $x \in \mathbb{N}^3$, consider the sets $C_x = C_{x_1, x_2, x_3} = C^{(1)}_{x_1} \cap C^{(2)}_{x_2} \cap C^{(3)}_{x_3}$. Let $X$ be the set of all $x$ such that $C_x\not=\emptyset$. If $G^{(\chi')}[4]$ (where the superscript indicates the relevant colouring) has an induced path $v_1, v_2, \dots, v_r$, then if $x_i \in \mathbb{N}^3$ is defined to be such that $v_i \in C_{x_i}$, in fact $x_1, x_2, \dots, x_r$ becomes an induced path in $G_3$. But Lemma~\ref{g3path} implies that $r \leq 30$. Hence, the 4-diameter in the colouring $\chi'$ is at most 30.\\
Applying Theorem~\ref{connThm} for the colouring $\chi'$, gives three monochromatic components that cover the vertex set, let these be $G^{(\chi')}[A_1, c_1], G^{(\chi')}[A_2, c_2], G^{(\chi')}[A_3, c_3]$, where the superscript indicates the relevant colouring. Using the same sets and colours, but returning to the original colouring, we have that $G^{(\chi)}[A_1, c_1], G^{(\chi)}[A_2, c_2], G^{(\chi)}[A_3, c_3]$ are all still connected, as 1, 2 and 3-components are the same in $\chi$ and $\chi'$, while there can only be more 4-coloured edges in the colouring $\chi$. Also, 1, 2 and 3-diameters are bounded by $N_1$, and 4-diameters of sets may only decrease when returning to colouring $\chi'$, so the lemma follows.\end{proof}

Let us introduce some additional notions. Let $P \subset \mathbb{N}_0^2$ be a set, and let $L\colon P \to \mathcal{P}(n)\setminus \{\emptyset\}$ be a function with the property that $\{L(A) \colon A \in P\}$ form a partition of $[n]$ and there a two colours $c_3, c_4$ \footnote{This choice of indices was chosen on purpose -- we shall first use colours $c_1, c_2$ to define $P$ and $L$, and the remaining colours will be $c_3$ and $c_4$.}such that whenever $A, B \in P$ and $|A_1 - B_1|, |A_2 - B_2| \geq 2$, then all edges between the sets $L(A)$ and $L(B)$ are coloured with $c_3$ and $c_4$ only. We call $L$ the \emph{$c_3, c_4$-layer mapping} and we refer to $P$ as the \emph{layer index set}. Further, we call a subset $S \subset P$ a \emph{$k$-distant set} if for every two distinct points $A, B \in S$ we have $|A_1 - B_1|,|A_2 - B_2| \geq k$.\\
Let us briefly motivate this notion. Suppose that $K_n[c_1]$ and $K_n[c_2]$ are both connected. Fix a vertex $x_0$ and let $P = \{(d_{c_1}(x_0, v), d_{c_2}(x_0, v))\colon v \in [n]\} \subset \mathbb{N}_0^2$. Let $L(A) \colon= \{v \in [n]\colon (d_{c_1}(x_0, v), d_{c_2}(x_0, v)) = A\}$ for all $A \in P$ (this also motivates the choice of the letter $L$, we think of $L(A)$ as a \emph{layer}). Then, if $x \in L(A), y \in L(B)$ for $A, B \in P$ with $|A_1 - B_1| \geq 2, |A_2 - B_2| \geq 2$, by triangle inequality, we cannot have $d_{c_1}(x,y) \leq 1$ nor $d_{c_2}(x,y) \leq 1$, so $xy$ takes either the colour $c_3$ or the colour $c_4$. As we shall see, we may have more freedom in the definition of $P$ and $L$ if there is more than one component in a single colour.\\
We now explore these notions in some detail, before using them to obtain some structural results on the 4-colourings that possibly do not satisfy Conjecture~\ref{bddConj}.

\begin{lemma}\label{dist3setLemma} Let $\chi$ be a 4-colouring, $L$ a \emph{$c_3, c_4$-layer mapping} with layer index set $P$, and suppose that $\{A, B, C\} \subset P$ is a $3$-distant set. Write $G = K_n$. Then the following hold.
\begin{enumerate}
\item For some colour $c \in \{c_3, c_4\}$ we have $G[L(A) \cup L(B) \cup L(C), c]$ connected and of diameter at most 20.
\item If additionally for $c'$ such that $\{c,c'\} = \{c_3,c_4\}$ and some distinct $A', B' \in \{A, B, C\}$ we have $G[L(A')\cup L(B'), c']$ contained in a subgraph $H \subset G[c']$ that is connected and of diameter at most $N_3$, then the given colouring satisfies Conjecture~\ref{bddConj} with $\max\{40, N_3 + 20\}$.
\end{enumerate} \end{lemma}

\begin{proof}[Proof of Lemma~\ref{dist3setLemma}.] 
\textbf{(1):} Observe that all edges between $L(A), L(B), L(C)$ are of colours $c_3$ and $c_4$. This is a complete tripartite graph and by Lemma~\ref{mult2col} w.l.o.g. $L(A) \cup L(B) \cup L(C)$ is $c_3$-connected and of $c_3$-diameter at most 20.\\
\textbf{(2):} W.l.o.g. $A' = A, B' = B$. Pick any $D \in P$. Note that since $A, B, C$ are 3-distant, $D$ is 2-distant from at least one of $A, B, C$ (otherwise, by pigeonhole principle, for some $A', B'$ among $A, B, C$ and some index $i$, we have $|A'_i - D_i|, |B'_i - D_i| \leq 1$, so $|A'_i - B'_i| \leq 2$, which is impossible). Let $E \in \{A, B, C\}$ be such that $D, E$ are $2$-distant. Thus, all the edges between $L(D)$ and $L(E)$ are of colours $c_3$ and $c_4$, so Lemma~\ref{2colsBip} applies to $L(D) \cup L(E)$.\\
Let $P' \subset P$ be the set of all $D \in P$ such that Lemma~\ref{2colsBip} gives that either $L(D) \cup L(E)$ is $c$-connected and of $c$-diameter at most 10, or the second conclusion of that lemma holds. Hence, every vertex $x$ in $L(D)$ for some $D \in P'$ is on $c$-distance at most 10 to a vertex in $L(A) \cup L(B) \cup L(C)$, which itself has $c$-diameter at most 20. Hence, $L(A) \cup L(B) \cup L(C) \cup (\cup_{D \in P'} L(D))$ is $c$-connected and of $c$-diameter at most 40.\\
For all other $D \in P\setminus P'$, Lemma~\ref{2colsBip} applied to $L(D) \cup L(E)$ for a relevant $E$ implies that $L(D) \cup L(E)$ is $c'$-connected and of diameter at most 10. Let $P''$ be the set of $D \in P\setminus P'$ for which $E \in \{A, B\}$, and let $P''' = P \setminus (P' \cup P'')$ (for which therefore $E = C$). Hence, $H \cup (\cup_{D \in P''} L(D))$ is $c'$-connected and of $c'$-diameter at most $N_3 + 20$, and finally $L(C)\cup (\cup_{D \in P'''} L(D))$ is also $c'$-connected and of $c'$-diameter at most 20. Hence, taking
\begin{equation*}\begin{gathered}
G[L(A) \cup L(B) \cup L(C) \cup (\cup_{D \in P'} L(D)), c],\\
H \cup G[(\cup_{D \in P''} L(D)), c'], \text{ and}\\
G[L(C)\cup (\cup_{D \in P'''} L(D)), c'],\\
\end{gathered}\end{equation*}
proves the lemma.\end{proof}

\begin{lemma} \label{3distLemma}Suppose that $\chi$ is a 4-colouring of $E(K_n)$ and that $L$ is a $c_3, c_4$-layer mapping for some colours $c_3, c_4 \in [4]$ with a $3$-distant set of size at least 4. Then $c$ satisfies Conjecture~\ref{bddConj} with constant 160. \end{lemma}

\begin{proof} Write $G = K_n$. Suppose that some $A, B, C, D \in P$ are 3-distant. All edges between $L(A) \cup L(B) \cup L(C) \cup L(D)$ are of colours $c_3$ and $c_4$ only, so by Lemma~\ref{mult2col} w.l.o.g. $G[L(A) \cup L(B) \cup L(C) \cup L(D),c_3]$ is connected and of diameter at most 60. Pick any $E \in P$. If $E$ has difference at most 1 in absolute value in some coordinate from at least three points among $A, B, C, D$, by pigeonhole princple, there are $A' , B'$ among these four and coordinate $i$ such that $|A'_i - E_i|, |B'_i - E_i| \leq 1$ so $|A'_i - B'_i| \leq 2$, which is impossible. Hence, $E$ is 2-distant from at least two points $A'(E), B'(E)$ among $A, B, C, D$. Hence, $A'(E), B'(E), E$ is a 2-distant set, so edges between $L(A'(E)), L(B'(E))$ and $L(E)$ are of colours $c_3$ and $c_4$ only. By Lemma~\ref{mult2col}, for some colour $c(E) \in \{c_3,c_4\}$ we have $G[L(A'(E)) \cup L(B'(E)) \cup L(E), c(E)]$ connected and of diameter at most 20. We split $P$ as follows: $P' \subset P$ is the set of all $E \in P$ such that $c(E) = c_3$, and for each pair $\pi$ of $A, B, C, D$ we define $P_\pi$ as the set of all $E \in P$ such that $\{A'(E), B'(E)\} = \pi$ and $c(E) = c_4$. We now look at the set of all pairs $\pi$ for which $P_\pi\not=\emptyset$.\\
\textbf{Case 1: there are $\pi_1, \pi_2$ such that $P_{\pi_1}$ and $P_{\pi_2}$ are non-empty and $\pi_1 \cap \pi_2 \not=\emptyset$.} W.l.o.g. $\pi_1 = \{A, B\}, \pi_2 = \{A,C\}$. For every $\pi = \{A', B'\}$ we already have $G[L(A') \cup L(B') \cup (\cup_{E \in P_\pi} L(E)), c_4]$ connected and of diameter at most 40. Hence, $G[L(A) \cup L(B) \cup L(C) \cup (\cup_{E \in P_{\pi_1} \cup P_{\pi_2}} L(E)), c_4]$ is also connected and of diameter at most 80. But, any other pair $\pi$ must intersect $A, B, C$, so we have
$$G[\cup_{\pi} \left((\cup_{F \in \pi} L(F)) \cup (\cup_{E \in P_{\pi}} L(E))\right), c_4]$$
connected and of diameter at most 160, where $\cup_\pi$ ranges over all pairs. Taking additionally
$$G[L(A) \cup L(B) \cup L(C) \cup L(D) \cup (\cup_{E \in P'} L(E)), c_3]$$
proves the claim.\\
\textbf{Case 2: all pairs $\pi$ such that $P_\pi \not= \emptyset$ are disjoint.} There are at most 2 such pairs. Thus, if we take 
$$G[(\cup_{F \in \pi} L(F)) \cup (\cup_{E \in P_{\pi}} L(E)), c_4]$$
for such pairs $\pi$ (these are connected and of diameter at most 40), and
$$G[L(A) \cup L(B) \cup L(C) \cup L(D) \cup (\cup_{E \in P'} L(E)), c_3],$$
the claim follows.
\end{proof}

\begin{lemma} \label{7distsize3}Suppose that $\chi$ is a 4-colouring of $E(K_n)$ and that $L$ is a $c_3, c_4$-layer mapping for some colours $c_3, c_4 \in [4]$ with a $7$-distant set of size at least 3. Suppose additionally that $\{A_i\colon A \in P\}$ takes at least 28 values for each $i=1,2$. Then $\chi$ satisfies Conjecture~\ref{bddConj} with constant 160. \end{lemma}

\begin{proof} Let $\{A, B, C\}$ be a 7-distant set. Pick any other $D \in P$. If $D$ is 3-distant from each of $A, B, C$, we obtain a 3-distant set of size 4, so by Lemma~\ref{3distLemma} we are done. Hence, for every $D \in P$ we have $E \in \{A, B, C\}$ such that $|E_i - D_i| \leq 2$ for some $i$. (Note that this is the main contribution to the constant 160 in the statement.)\\
Since $\{A, B, C\}$ is a 7-distant set, by Lemma~\ref{dist3setLemma}, we have w.l.o.g. $G[L(A) \cup L(B) \cup L(C), c_3]$ connected and of diameter at most 20. We now derive some properties of $L(D)$ for points $D \in P$ be such that $|D_i - A_i|, |D_i - B_i|, |D_i - C_i| \geq 3$ for some $i \in \{1,2\}$. (Note that such points exist by assumptions.)\\
\indent Let $D$ be such a point and let $j$ be such that $\{i, j\} = \{1,2\}$. Since the set $\{A, B, C\}$ is $7$-distant, there are distinct $E_1, E_2 \in \{A, B, C\}$ such that $|D_j - (E_1)_j|, |D_j - (E_2)_j| \geq 3$. Thus, $\{D, E_1, E_2\}$ is also a 3-distant set. Applying Lemma~\ref{dist3setLemma} to $\{D, E_1, E_2\}$ implies that $G[L(D) \cup L(E_1) \cup L(E_2), c]$ is connected and of diameter at most 20, for some $c \in \{c_3, c_4\}$. However, if $c = c_4$, $G[L(E_1) \cup L(E_2), c_4]$ is contained in a subgraph of $G[c_4]$ that is connected and of diameter at most 20, so Lemma~\ref{dist3setLemma} (2) applies once again and the claim follows. Hence, we must have $G[L(D) \cup L(E_1) \cup L(E_2), c_3]$ is connected and of diameter at most 20. In particular, whenever $D \in P$ satisfies $|D_i - A_i|, |D_i - B_i|, |D_i - C_i| \geq 3$ for some $i \in \{1,2\}$, then every point in $L(D)$ is on $c_3$-distance at most 20 from $L(A) \cup L(B) \cup L(C)$.\\

By assumptions $\{A_1\colon A \in P\}$ takes at least 28 values. Hence, we can find $X \in P$ such that $|X_1 - A_1|, |X_1 - B_1|, |X_1 - C_1| \geq 5$. Similarly, there is $Y \in P$ such that $|Y_2 - A_2|, |Y_2 - B_2|, |Y_2 - C_2| \geq 5$. W.l.o.g $|X_2 - A_2| \leq 2$. If $|Y_1 - A_1| \leq 2$, then $X, Y, B, C$ form a 3-distant set of size 4, and once again the claim follows from Lemma~\ref{3distLemma}. Hence, w.l.o.g. $|Y_1 - B_1| \leq 2$. By the work above, we also have that every point in $L(X) \cup L(Y)$ is on $c_3$-distance at most 20 from $L(A) \cup L(B) \cup L(C)$. Note also that $X, Y$ are 3-distant.\\

It remains to analyse $D \in P$ such that for both $i = 1,2$ there is an $E \in \{A, B, C\}$ such that $|E_i - D_i| \leq 2$. We show that in all but one case on the choice of sets $E$, we in fact have $L(D)$ on bounded $c_3$-distance to $L(A) \cup L(B) \cup L(C)$. If we have an $E \in \{A, B, C\}$ such that both $|E_1 - D_1| \leq 2$ and $|E_2 - D_2| \leq 2$ hold, then taking $E', E''$ such that $\{E, E', E''\} = \{A, B, C\}$, we have $D, E', E''$ 3-distant, so Lemma~\ref{dist3setLemma} once again implies that every vertex in $L(D)$ is on $c_3$-distance at most 20 from $L(A) \cup L(B) \cup L(C)$ (or we are done by the second part of Lemma~\ref{dist3setLemma}).\\
We distinguish the following cases.
\begin{itemize}
\item If $|D_1 - A_1| \leq 2, |D_2 - B_2| \leq 2$, then $D, X, Y$ form a 3-distant set. Let us check this. We already have $X, Y$ 2-distant. By triangle inequality, we obtain $|X_1 - D_1| \geq |X_1 - A_1| - |A_1 - D_1| \geq 3$, $|Y_1 - D_1| \geq |B_1 - A_1| - |B_1 - Y_1| - |D_1 - A_1| \geq 3$, $|D_2 - X_2| \geq |B_2 - A_2| - |B_2 - D_2| - |X_2 - A_2| \geq 3$ and $|Y_2 - D_2| \geq |Y_2 - B_2| - |B_2 - D_2| \geq 3$.\\
We also know that $L(X) \cup L(Y)$ is contained in a subgraph $H \subset G[c_3]$ that is connected and of diameter at most 20, so applying Lemma~\ref{dist3setLemma} implies that we are done, unless $G[L(D) \cup L(X) \cup L(Y), c_3]$ is connected and of diameter at most 20. Hence $L(D)$ is on $c_3$-distance at most 40 from $L(A) \cup L(B) \cup L(C)$.
\item If $|D_1 - C_1| \leq 2, |D_2 - B_2| \leq 2$, then the same argument we had in the case above proves that $L(D)$ is on $c_3$-distance at most 40 from $L(A) \cup L(B) \cup L(C)$.
\item If $|D_1 - A_1| \leq 2, |D_2 - C_2| \leq 2$, then the same argument we had in the case above proves that $L(D)$ is on $c_3$-distance at most 40 from $L(A) \cup L(B) \cup L(C)$.
\end{itemize} 
Finally, we define $P_1, P_2, P_3 \subset P$ as
\begin{equation*}\begin{split}
P_1 &= \{D \in P\colon |D_1 - B_1|, |D_2 - A_2| \leq 2\}\\
P_2 &= \{D \in P\colon |D_1 - C_1|, |D_2 - A_2| \leq 2 \}\\
P_3 &= \{D \in P\colon |D_1 - B_1|, |D_2 - C_2| \leq 2 \}
\end{split}\end{equation*}
which are disjoint and if $D \in P \setminus (P_1 \cup P_2 \cup P_3)$ we know that $L(D)$ is on $c_3$-distance at most 40 from $L(A) \cup L(B) \cup L(C)$. Let also $L_i = \cup_{D \in P_i} L(D)$. Hence, since for $D \in P_1$ we have $|D_1 - C_1|, |D_2 - C_2| \geq 2$, all edges between $L(D)$ and $L(C)$ are coloured using $c_3$ and $c_4$, we actually have all edges between $L_1$ and $L(C)$ coloured using only these two colours. Applying Lemma~\ref{2colsBip} we have $G[L_1 \cup L(C), c]$ connected and of diameter at most 10 for some $c \in \{c_3, c_4\}$, or $L_1$ is on $c_3$-distance 1 from $L(A) \cup L(B) \cup L(C)$. Similarly, all edges between $L_2$ and $Y$, and all edges between $L_3$ and $X$ are taking only the colours $c_3$ and $c_4$. Observe that if $D \in P_2, D' \in P_3$ then $|D_1 - D'_1| \geq |C_1 - B_1| - |C_1 - D_1| - |D'_1 - B_1| \geq 3$. Similarly, $|D_2 - D'_2| \geq |A_2 - C_2| - |A_2 - D_2| - |D_2' - C_2|  \geq 3$, so all edges between $L_2$ and $L_3$ are only of colours $c_3$ and $c_4$. Apply Lemma~\ref{2colsBip} to $L_2$ and $L(Y)$, implying either $G[L_2 \cup L(Y), c_4]$ is connected and of diameter at most 10, or $L_2$ is on $c_3$-distance at most 30 from $L(A) \cup L(B) \cup L(C)$. Similarly, apply Lemma~\ref{2colsBip} to $L_3$ and $L(X)$, implying either $G[L_3 \cup L(X), c_4]$ is connected and of diameter at most 10, or $L_3$ is on $c_3$-distance at most 30 from $L(A) \cup L(B) \cup L(C)$. Finally, let $V = \{v \in [n] \colon d_{c_3}(v, L(A) \cup L(B) \cup L(C)) \leq 40\}$, which is $c_3$-connected and of $c_3$-diameter at most 100. We distinguish the following cases.
\begin{itemize}
\item $L_2, L_3 \subset V$. In this case, we can take $V$ and $G[L_1 \cup L(C), c]$ if necessary (otherwise $L_1 \subset V$).
\item $L_2 \not\subset V, L_3 \subset V$. Thus, $G[L_2 \cup L(Y), c_4]$ is connected and of diameter at most 10, so taking $G[L_2 \cup L(Y), c_4]$ and $V$, and additionally $G[L_1 \cup L(C), c]$ if necessary, we are done.  
\item $L_2 \subset V, L_3 \not\subset V$. Thus, $G[L_3 \cup L(X), c_4]$ is connected and of diameter at most 10, so taking $G[L_3 \cup L(X), c_4]$ and $V$, and additionally $G[L_1 \cup L(C), c]$ if necessary, we are done. 
\item $L_2, L_3 \not\subset V.$ In this case, we have $G[L_2 \cup L(Y), c_4]$ and $G[L_3 \cup L(X), c_4]$ connected and of diameter at most 10. Apply Lemma~\ref{2colsBip} to $L_2$ and $L_3$. If $L_2$ and $L_3$ are on $c_4$-distance at most 10, we may take $G[L_2 \cup L_3 \cup L(X) \cup L(Y), c_4]$, $V$ and $G[L_1 \cup L(C), c]$ if necessary. Otherwise, we have $G[L_2 \cup L_3, c_3]$ is connected and of diameter at most 10. In this case, take $G[L_2 \cup L_3, c_3], V$ and $G[L_1 \cup L(C), c]$ if necessary.
\end{itemize}
This completes the proof of the lemma.
\end{proof}

Let us now briefly discuss a way of defining $c_3, c_4$-layer mappings. Pick two colours $c_1, c_2 \in [4]$, and take $c_3, c_4$ to be the remaining two colours. List all the vertices as $v_1, v_2, \dots, v_n$. To each vertex, we shall assign two nonnegative integers, $D_1(v_i)$ and $D_2(v_i)$, initially marked as undefined. We apply the following procedure.
\begin{itemize}
\item[\textbf{Step 1}] Pick the smallest index $i$ such that $D_1(v_i)$ or $D_2(v_i)$ is undefined. If there is no such $i$, terminate the procedure.
\item[\textbf{Step 2}] For $j = 1,2$, if  $D_j(v_i)$ is undefined, pick an arbitrary value for it.
\item[\textbf{Step 3}] For $j = 1,2$, if  $D_j(v_i)$ was undefined before the second step, for all vertices $u$ in the same $c_j$-component of $v_i$ set $D_j(u) \colon= d_{c_j}(v_i, u) + D_j(v_i)$. Return to Step 1.    
\end{itemize}
Upon the completion of the procedure, set $P = \{(D_1(v), D_2(v))\colon v \in [n]\}$ and $L\colon P \to \mathcal{P}(n)$ as $L(x,y) \colon= \{v \in [n] \colon (D_1(v), D_2(v)) = (x,y)\}$.\\

\noindent\textbf{Claim.} The mapping $L$ above is well-defined and is a $c_3,c_4$-layer mapping.
\begin{proof}Observe that each time we pick $v_i$ whose value(s) are to be defined, we end up defining $D_1$ on one $c_1$-component or $D_2$ on one $c_2$-component or both. Hence, for every vertex $v$, the values $D_1(v), D_2(v)$ change precisely once from undefined to a nonnegative integer value. Hence, $(D_1(v), D_2(v))$ are well-defined and take values in $\mathbb{N}_0^2$, so $P$ and $L$ are well-defined and $L(A)$ forms a partition of $[n]$ as $A$ ranges over $P$. Finally, consider an edge $xy$ coloured by $c_1$. Let $D_1(x)$ be defined with $v_i$ chosen in Step 2 (possibly $x = v_i$). Since $xy$ is of colour $c_1$, these are in the same $c_1$-component, and hence $D_1(x) = d_{c_1}(v_i, x) + D_1(v_i)$ and $D_1(y) = d_{c_1}(v_i, y) + D_1(v_i)$. Therefore,
$$|D_1(x) - D_1(y)|  = |(d_{c_1}(v_i, x) + D_1(v_i)) - (d_{c_1}(v_i, y) + D_1(v_i))| = |d_{c_1}(v_i, x) - d_{c_1}(v_i, y)| \leq d_{c_1}(x,y) = 1$$
hence, if $\chi(xy) = c_1$, then $|D_1(x) - D_1(y)| \leq1$. Similarly, we get the corresponding statement for the colour $c_2$. It follows that if $A, B \in P$ are such that $|A_1 - B_1|, |A_2 - B_2| \geq 2$, then if $x \in L(A), y \in L(B)$, we have $(D_1(x), D_2(x)) = A, (D_1(y), D_2(y)) = B$, so $xy$ is coloured by $c_3$ or $c_4$, as desired.\end{proof}

\subsection{Monochromaticly connected case}

\begin{proposition}\label{singleComp}Suppose that $\chi$ is a 4-colouring of $E(K_n)$ such that every colour induces a connected subgraph of $K_n$. Then $\chi$ satisfies Conjecture~\ref{bddConj} with constant 160. \end{proposition}

\begin{proof} Suppose contrary, in particular every colour has diameter greater than 480. Our main goal in the proof is to find a pair of vertices $x',y'$ with a control over their 1-distance and 2-distance. We need both distances sufficiently large so that we can make a use of distant sets in $3,4$-layer mappings, and also bounded by a constant so that if a vertex is on small 1-distance from $x'$, it is also on small 1-distance from $y'$ and vice-versa.\\

More precisely,

\begin{lemma}\label{contrL}Suppose that there are vertices $x', y'$ such that $d_1(x', y') \in \{6,7, \dots, 50\}, d_2(x', y') \in \{10, 11, \dots, 20\}$. Then we obtain a contradiction.\end{lemma}

\begin{proof} Pick any point $z \not= x', y'$. Apply the procedure for defnining $3,4$-layer mapping starting from $x'$. If we obtain a 7-distant set of size at least 3, we obtain a contradiction with Lemma~\ref{7distsize3}. Hence, the distances corresponding to $x', y', z$ cannot give such a set, so we must have one of
$$d_1(x', z) \leq 6\text{ or }|d_1(x', y') - d_1(x', z)| \leq 6\text{ or }d_2(x', z) \leq 6\text{ or }|d_2(x', z) - d_2(x', y)| \leq 6.$$
In particular, we must have $d_1(x', z) \leq 56$ or $d_2(x', z) \leq 26$. Recalling the definition of monochromatic balls, $B_1(x, 56)$ and $B_2(x, 26)$ cover all the vertices, giving a contradiction.\end{proof}

\noindent\textbf{Claim.} There are $x,y$ such that $d_1(x,y) \in \{25,26,27\}$ and $d_2(x,y) \geq 40$.

\begin{proof}[Proof of the claim.] Suppose contrary, for every $x,y$ such that $d_1(x,y) \in \{25,26,27\}$, we must have $d_2(x,y) \leq 39$. Pick any $y_1, y_2 \in [n]$ such that $\chi(y_1y_2) = 1$. Since the 1-diameter is greater than 160, we can find $x \in [n]$ such that $d_1(x,y_1) = 26$. By triangle inequality, we also have $d_1(x,y_2) \in \{25,26,27\}$. Hence, $d_2(x,y_1), d_2(x,y_2) \leq 39$, from which we conclude that whenever an edge $y_1y_2$ is coloured by 1, then $d_2(y_1, y_2) \leq 78$. Hence, taking any $x \in [n]$ the balls
$$B_2(x, 78), B_3(x, 1), B_4(x, 1)$$
cover the vertex set. However, these have diameter less than 160, which is a contradiction.\end{proof}

Take $x,y$ given by the claim above. Since the subgraph $G[2]$ is connected, there is a minimal 2-path $x = z_0, z_1, \dots, z_{r}, z_{r+1} = y$ between $x$ and $y$, with $r \geq 39$. Look at the vertices $z_{10}, z_{20}, \dots, z_{10k}$ with $k$ such that $10 \leq r - 10k < 20$.\\
Consider $x, y, z_{10i}$ for some $1 \leq i \leq k$ and check whether we can define $3,4$-layer mapping so that these three points become a 7-distant set. Apply the procedure for defining $3,4$-layers mapping, starting from $x$, i.e. we want to see whether $(0,0), (d_1(x,y), d_2(x,y))$ and $(d_1(x, z_{10i}), d_2(x,z_{10i}))$ are 7-distant. If they are 7-distant, Lemma~\ref{7distsize3} gives us a contradiction.  Since 
\begin{equation*} \begin{split}
d_1(x,y) \geq 25, d_2(x,y) \geq 39\\
10 \leq d_2(x, z_{10i}) = 10i \leq 10k < d_2(x,y) - 6\\
\end{split}\end{equation*}
we must have either $d_1(x,z_{10i}) \leq 6$ or $|d_1(x, z_{10i}) - d_1(x,y)| \leq 6$ (implying $d_1(x, z_{10i}) \in \{19, 20, \dots, 33\}$). Similarly, if we start from $y$ instead of $x$ in our procedure, we see that either $d_1(y, z_{10i}) \leq 6$ or $|d_1(y, z_{10i}) - d_1(x,y)| \leq 6$ (implying $d_1(y, z_{10i}) \in \{19, 20, \dots, 33\}$) must hold.\\
Observe that for the vertex $z_{10}$ we must have $d_1(x, z_{10}) \leq 6$. Otherwise, we would have $19 \leq d_1(x, z_{10}) \leq 33$ and $d_2(x, z_{10}) = 10$, resulting in a contradiction by Lemma~\ref{contrL} (applied to the pair $x, z_{10}$). For every $z_{10i}$ we must have either the first inequality ($d_1(x,z_{10i}) \leq 6$) or the second ($19 \leq d_1(x,z_{10i}) \leq 33$), and we have that the first vertex among these, namely $z_{10}$, satisfies the first inequality. Suppose that there was an index $i$ such that $z_{10(i+1)}$ obeys the second inequality, and pick the smallest such $i$. Then, by the triangle inequality, we would have
$$13\leq d_1(z_{10(i+1)}, x) - d_1(x, z_{10i}) \leq d_1(z_{10i}, z_{10(i+1)}) \leq d_1(z_{10(i+1)}, x) + d_1(x, z_{10i}) \leq 39$$
and $d_2(z_{10i}, z_{10(i+1)}) = 10$, so Lemma~\ref{contrL} applies now to the pair $z_{10i}, z_{10(i+1)}$ and gives a contradiction. Hence, for all $i \leq k$ we must have the first inequality for $z_{10i}$. But then $z_{10k}$ and $y$ satisfy the conditions of Lemma~\ref{contrL}, giving the final contradiction, since $10 \leq d_2(y, z_{10k}) < 20$ and
$$19 \leq d_1(y, x) - d_1(x, z_{10k}) \leq d_1(y, z_{10k}) \leq d_1(y, x) + d_1(x, z_{10k}) \leq 33.$$
This completes the proof.\end{proof}

\subsection{Intersecting monochromatic components}

\begin{proposition}\label{intersectingcmps} Suppose that $\chi\colon E(K_n) \to  [4]$ be a 4-colouring with the property that, whenever $C$ and $C'$ are monochromatic components of different colours, and one of them has diameter at least 30 (in the relevant colour), then $C$ and $C'$ intersect. Then $\chi$ satisfies Conjecture~\ref{bddConj} with constant 160.
\end{proposition}
\begin{proof} Suppose contrary, we have a colouring $\chi$ that satisfies the assumptions but for which the conclusion fails. By Lemma~\ref{3smalldiam}, we have that at least two colours have monochromatic diameters greater than 160. Let $C_1$ be such a component for colour $c_1$, and let $C_2$ be such a component for colour $c_2$, with $c_1 \not= c_2$. Further, by the Proposition~\ref{singleComp} we have a colour $c'$ (which might equal one of $c_1, c_2$) with at least two components, w.l.o.g. $c_1 \not= c'$.\\
First, we find a pair of vertices $x,y$ with the property that $10 \leq d_{c_1}(x,y) \leq 40$ and $x, y$ are in different $c'$-components. We do this as follows. If there are a couple of vertices $x_1, x_2$ with $d_{c_1}(x_1, x_2) < 10$ that are in different $c'$-components, then, since $c_1$-diameter of $C_1$ is large, we can find $y \in C_1$ with $d_{c_1}(x_1, y) = 25$. Hence, $15 \leq d_{c_1}(x_2, y) \leq 35$, and $y$ is in different $c'$-component from one of $x_1, x_2$, yielding the desired pair. Otherwise, we have that all pairs of vertices $x,y \in C_1$ with $d_{c_1}(x,y) \leq 30$ also share the same $c'$-component. But then, we must have the whole $c_1$-component $C_1$ contained in one $c'$-component, making it unable to intersect other $c'$-components, which is impossible. Hence, we have $x, y$ in different $c'$-components, with $10 \leq d_{c_1}(x,y) \leq 40$.\\
Pick any vertex $z$ outside $B_{c_1}(x, 50)$. Let $c'', c'''$ be the two colours different from $c_1, c'$. We now apply our procedure for defining $c'', c'''$-layers mapping with vertices listed as $x, y, z, \dots$. Note that $|D_1(x) - D_1(y)|, |D_1(x) - D_1(z)|, |D_1(y) - D_1(z)| \geq 10$ (recall the $D_1, D_2$ notation from the procedure). Hence, we get a 7-distant set, unless $d_{c'}(x,z) \leq 6$ or $d_{c'}(y,z) \leq 6$. Hence, $B_{c_1}(x, 50)$, $B_{c'}(x, 6)$ and $B_{c'}(y, 6)$ cover the vertex set and we get a contradiction.
\end{proof}

\subsection{Final steps}
In the final part of the proof, we show how to reduce the general case to the case of intersecting monochromatic components.

\begin{theorem}Conjecture~\ref{bddConj} holds for 4 colours and we may take 160 for the diameter bounds.\end{theorem}
\begin{proof} Let $\chi$ be the given $4$-colouring of $E(K_n)$. Our goal is to apply the Proposition~\ref{intersectingcmps}. We start with an observation.

\begin{observation}\label{disjcmps} Suppose that $C$ is a $c$-component, that is disjoint from a $c'$-component $C'$ with $c' \not=c$. Then for every pair of vertices $x,y \in C$ we have $d_c(x,y) \leq 6$ or $d_{c'}(x,y) \leq 6$ or the colouring satisfies Conjecture~\ref{bddConj} with the constant 160. \end{observation}

\begin{proof}[Proof of the Observation~\ref{disjcmps}] Pick $x,y \in C$ with $d_c(x,y) \geq 7$ and take arbitrary $z \in C'$. Apply our procedure for generating $c_3, c_4$-layers mapping to the list $x, y, z, \dots$, with $c_3, c_4$ chosen to be the two colours different from $c, c'$. Since $z$ is in different $c$- and $c'$-components from $x, y$, these three vertices result in a 7-distant set, unless $d_{c'}(x,y) \leq 6$, as desired.\end{proof}

\begin{corollary} Suppose that we have a $c$-component $C$, that is disjoint from a $c'$-component $C'$ with $c' \not=c$ and has $c$-diameter at least 30. Then the colouring $\chi$ satisfies Conjecture~\ref{bddConj} with the constant 160.\end{corollary}

\begin{proof} By the Observation~\ref{disjcmps} we are either done, or any two vertices $x,y \in C$ with $d_c(x,y) > 6$ satisfy $d_{c'}(x,y) \leq 6$. Furthermore, given any two vertices $x,y \in C$, since the $c$-diameter of $C$ is at least 30, we can find $z \in C$ such that $d_{c}(x,z), d_c(y,z) \geq 7$, so by triangle inequality $d_{c'}(x,y) \leq 12$ holds for all $x,y \in C$.\\
Now, take an arbitrary vertex $v \in C$, let $c'', c'''$ be the two remaining colours, and consider the sets
$$B_{c'}(v, 12), B_{c''}(v,1), B_{c'''}(v,1).$$
Given any $u \in [n]$, if $vu$ is coloured by any of $c', c''$ or $c'''$, it is already in the sets above. On the other hand, if $uv$ is of colour $c$, then $v \in C$ so $d_{c'}(u,v) \leq 10$, thus $u \in B^{(c')}(v, 10)$. Thus, these sets cover the vertex sets and have monochromatic diameters at most 24, so we are done.\end{proof}

Finally, we are in the position to apply the Proposition~\ref{intersectingcmps} which finishes the proof of the theorem.\end{proof}

\section{Concluding remarks}

Apart from the main conjectures~\ref{connConj} (and its equivalent~\ref{altConj}) and~\ref{bddConj}, here we pose further questions. Recall the section 2 that contains the auxiliary results. There we first discussed Lemmas~\ref{2colsBip} and~\ref{mult2col}, which were variants of the main conjectures with different underlying graph instead of $K_n$. Recall that Lovasz-Ryser conjecture is also about different underlying graphs. Another natural question would be the following.

\begin{question} Let $G$ be a graph, and let $k$ be fixed. Suppose that $\chi\colon E(G)\to[k]$ is a $k$-colouring of the edges of $G$. For which $G$ is it possible to find $k-1$ monochromaticly connected sets that cover the vertices of $G$? What bounds on their diameter can we take?\end{question}

Observe already that for 3 colours, the situation becomes much more complicated than that for 2 colours, where complete multipartite graphs behaved well. Consider the following example.\\

Pick $n+6$ vertices labelled as $v_1, v_2, \dots, v_6$ and $u_1, u_2, \dots, u_n$. Define the graph $G$ to be the complete graph on these vertices with 3 edges $v_1 v_2, v_3v_4$ and $v_5v_6$ removed. Define the colouring $\chi\colon E(G) \to[3]$ as follows.
\begin{itemize}
\item Edges of colour 1 are $v_1v_3, v_3v_5, v_1v_5, v_4v_6$ and $v_1u_i, v_3u_i, v_5u_i$ for all $i$.
\item Edges of colour 2 are $v_2v_4, v_2v_5, v_4v_5, v_1v_6$ and $v_2u_i, v_4u_i$ for all $i$.
\item Edges of colour 3 are $v_2v_3, v_2v_6, v_3v_6, v_1v_4$ and $v_6u_i$ for all $i$.
\item Edges of the form $u_i u_j$ are coloured arbitrarily.
\end{itemize}
\begin{figure}
\centering
\includegraphics[width=0.8\textwidth]{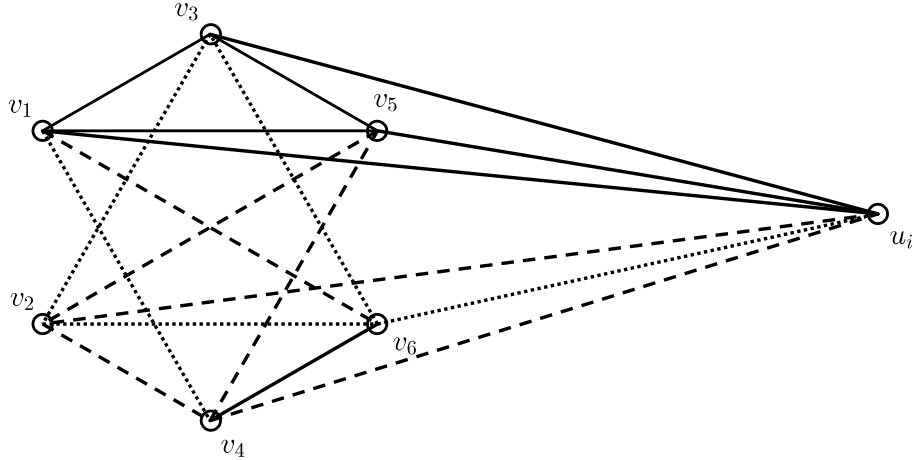}
\caption{An example of 3-colouring of $K_n$ with a matching of size 3 removed that cannot be covered by two monochromatic components.}\end{figure}
It is easy to check that this colouring has no covering of vertices by two monochromatic components. Is this essentially the only way the conjecture might fail for such a graph?

\begin{question}Let $G = K_n \setminus \{e_1, e_2, e_3\}$ be the complete graph with a mathching of size three omitted. Suppose that $\chi\colon E(G) \to [3]$ is a 3-colouring of the edges such that no two monochromatic components cover $G$. Is such a colouring isomorphic to an example similar to the one above? What about $K_{2n}$ with a perfect matching removed?\end{question}

Finally, recall that the one of the main contributions in the final bound in Theorem~\ref{mainthmbdd} came from Lemma~\ref{g3path} and that in general the Ramsey approach of Lemma~\ref{genSparse} would give much worse value. It would be interesting to study the right bounds for this problem as well.

\begin{question}For fixed $l$, what is the maximal size of a set of vertices $S$ of $G_l$ such that $G_l[S]$ is a path? What about other families of graphs of bounded degree? In particular, for fixed $l$ and $d$, what is the maximal size of a set of vertices $S$ of $G_l$ such that $G_l[S]$ is a connected graph of degrees bounded by $d$?\end{question}

\subsection{Acknowledgements}
I would like to thank Trinity College and the Department of Pure Mathematics and Mathematical Statistics of Cambridge University for their generous support. I am particularly indebted to Andr\'as Gy\'arf\'as and Imre Leader for the helpful discussions concerning this paper.

\end{document}